\documentclass[11pt]{article}
\title{Lower bounds for weak epsilon-nets and stair-convexity\footnote{A preliminary version
of this paper appeared in \emph{Proceedings of the 25th Annual
Symposium on Computational Geometry (SoCG 2009)}, ACM, New York, 2009, pp.~1--10.}}

\author{Boris Bukh\footnote{\texttt{B.Bukh@dpmms.cam.ac.uk}. Centre for Mathematical Sciences,
Cambridge CB3 0WB, England
and Churchill College, Cambridge CB3 0DS, England.
Work was carried out while the author was at Princeton University and at the R\'enyi Institute.
Work was partially supported by the project Phenomena in High Dimensions.}\and Ji\v
r\'i Matou\v sek\footnote{\texttt{matousek@kam.mff.cuni.cz}.
Department of Applied Mathematics and Institute of Theoretical
Computer Science (ITI), Charles University, Malostransk\'{e}
n\'{a}m. 25, 118~00~~Praha~1, Czech Republic.}\and Gabriel
Nivasch\footnote{\texttt{gabriel.nivasch@inf.ethz.ch}. Institute of
Theoretical Computer Science, ETH Z\"urich, CH-8092 Z\"urich, Switzerland.
Work was done while the author was at Tel Aviv University. Work was supported
by Israel Science Foundation Grant 155/05 and by the Hermann
Minkowski--MINERVA Center for Geometry at Tel Aviv University.}}

\date{}

\usepackage{amsmath,amsfonts,amssymb,amsthm}
\usepackage{hyperref}
\usepackage{graphics}

\newcommand*{\abs}[1]{\lvert #1\rvert} 
\newcommand{\R}{\mathbb{R}}  
\newcommand{\z}{\mathcal Z}  
\newcommand{\Gs}{G_{\mathrm s}} 
\newcommand{\Gu}{G_{\mathrm u}} 
\newcommand{\I}{[0,1]} 
\newcommand\myll[1]{\mathrel{\ll_{#1}}}    
\newcommand\mygg[1]{\mathrel{\gg_{#1}}}    
\newcommand{\B}{\mathrm{BB}(\Gs)}    
\newcommand{\diag}{D_{\mathrm{s}}}          
\DeclareMathOperator{\conv}{\mathrm{conv}}   
\DeclareMathOperator{\sconv}{\mathrm{stconv}}   
\DeclareMathOperator{\vol}{\mathrm{vol}}   
\DeclareMathOperator{\tp}{top}

\newcommand\FF{{\mathcal F}}

\newcommand\eps{\varepsilon}

\newtheorem{theorem}{Theorem}[section]
\newtheorem{lemma}[theorem]{Lemma}
\newtheorem{corollary}[theorem]{Corollary}
\newtheorem{prop}[theorem]{Proposition}

\newenvironment{remark}[1][Remark]
{\begin{trivlist}\item\textbf{#1:\ }} {\end{trivlist}}

\makeatletter
\newcommand{\ProofEndBox}{{\ifhmode\unskip\nobreak\hfil\penalty50 \else
          \leavevmode\fi\quad\vadjust{}\nobreak\hfill$\Box$
            \finalhyphendemerits=0 \par}}%
\makeatother
\newcommand{\proofend}{\ProofEndBox\smallskip}

\begin{document}

\maketitle

\begin{abstract}
A set $N\subset\R^d$ is called a \emph{weak $\eps$-net} (with
respect to convex sets) for a finite $X\subset\R^d$ if $N$
intersects every convex set $C$ with $|X\cap C|\ge\eps|X|$. For
every fixed $d\ge 2$ and every $r\ge 1$ we construct sets
$X\subset\R^d$ for which every weak $\frac 1r$-net has at least
$\Omega(r\log^{d-1} r)$ points; this is the first superlinear lower
bound for weak $\eps$-nets in a fixed dimension.

The construction is a \emph{stretched grid}, i.e., the Cartesian
product of $d$ suitable fast-growing finite sequences, and convexity
in this grid can be analyzed using \emph{stair-convexity}, a new
variant of the usual notion of convexity.

We also consider weak $\eps$-nets for the diagonal of our
stretched grid in $\R^d$, $d\ge 3$, which is an ``intrinsically
$1$-dimensional'' point set. In this case we exhibit slightly
superlinear lower bounds (involving the inverse Ackermann
function), showing that the upper bounds by Alon, Kaplan,
Nivasch, Sharir, and Smorodinsky (2008) are not far from the
truth in the worst case.

Using the stretched grid we also improve the known upper bound for
the so-called \emph{second selection lemma} in the plane by a
logarithmic factor: We obtain a set $T$ of $t$ triangles with
vertices in an $n$-point set in the plane such that no point is
contained in more than   $O\bigl( t^2 / (n^3 \log \frac {n^3}t )
\bigr)$ triangles of~$T$.

\end{abstract}

\section{Introduction and statement of results}\label{sec_intro}

\paragraph{\boldmath Weak $\eps$-nets. }
Let $X\subset\R^d$ be a finite set. A set $N\subseteq \R^d$ is
called a \emph{weak $\eps$-net\footnote{More precisely we
should say ``weak $\eps$-net for $X$ \emph{with respect to
convex sets}'', since later on we will also consider
$\eps$-nets with respect to another family of sets. But since
weak $\eps$-nets w.r.t.\ convex sets are our main object of
interest, we reserve the term ``weak $\eps$-net'' without
further specifications for this particular case.} for $X$},
where $\eps\in (0,1]$ is a real number, if $N$ intersects every
convex set $C$ with $|X\cap C|\ge\eps|X|$. This important
notion was introduced by Haussler and Welzl \cite{HW} and later
used in several results in discrete geometry, most notably in
the spectacular proof of the Hadwiger--Debrunner
$(p,q)$-conjecture by Alon and Kleitman \cite{ak-pcs-92a}. We
refer to Matou\v{s}ek \cite[Chap.~10]{matou_DG} for wider
background on weak $\eps$-nets and the related notion of
(strong) $\eps$-nets, and to Alon, Kalai, Matou\v{s}ek, and
Meshulam \cite{AKMM01} for a study of weak $\eps$-nets in an
abstract setting.

In this paper, instead of $\eps$, we will often use the parameter
$r:=\frac 1\eps\ge 1$ and thus speak of weak $\frac1r$-nets.

Let $f(X,r)$ denote the minimum cardinality of a weak $\frac1r$-net
for $X$. It is a nontrivial fact, first proved by Alon, B\'ar\'any,
F\"uredi, and Kleitman \cite{ABFK}, that
\begin{equation*}
f(d,r) := \sup\{f(X,r):X\subset\R^d\textrm{ finite}\}
\end{equation*}
is finite for every $d\ge 1$ and every $r\ge 1$; that is, for every
set $X$ there exist weak $\eps$-nets of size bounded solely in terms
of $d$ and $\eps$.

Several papers were devoted to estimating the order of magnitude of
$f(d,r)$. For $d=2$, the best upper bound is $f(2,r)=O(r^2)$
\cite{ABFK} (also see \cite{CEGGSW} for another proof), and for
every fixed $d\ge 3$ it is known that $f(d,r)=O(r^d (\log r)^{c_d})$
for some constant $c_d$ (Chazelle, Edelsbrunner, Grigni, Guibas,
Sharir, and Welzl \cite{CEGGSW}; also see Matou\v{s}ek and Wagner
\cite{MW} for a simpler proof).

The only known nontrivial lower bound for $f(d,r)$ asserts that
$f(d,50)=\Omega(\exp(\sqrt{d/2}))$ \cite{matou_lower}. It concerns
the dependence of $f(d,r)$ on $d$, and no lower bound, except for
the obvious estimate $f(d,r)\ge r$, has been known for $d$ fixed and
$r$ large. Our main result is a superlinear lower bound for every
fixed $d$.

\begin{theorem}\label{t:wen-lwb}
Let $d\ge 2$ be fixed. Then for  every $r\ge 1$ there exists a
finite set $\Gs\subset\R^d$ (a \emph{stretched grid}) such that
\begin{equation*}
f(\Gs,r)=\Omega(r\log^{d-1} r).
\end{equation*}
\end{theorem}

\paragraph{The stretched grid. }
The stretched grid $\Gs$ in the theorem is the Cartesian product
$X_1\times X_2\times \cdots\times X_d$, where each $X_i$ is a
suitable set of $m$ real numbers. The integer $m$ is a parameter of
the construction of $\Gs$, so we sometimes write $\Gs=\Gs(m)$, and
$m$ has to be chosen sufficiently large in terms of $r$ and~$d$ in
the proof of Theorem~\ref{t:wen-lwb}.

The main idea in the construction of $\Gs$ is that
$X_2,X_3,\ldots,X_d$ are ``fast-growing'' sequences, and each $X_i$
grows much faster than $X_{i-1}$. For technical reasons, we will not
define $\Gs(m)$ uniquely; rather, we will introduce some condition
that the $X_i$ have to satisfy, and thus, formally speaking,
$\Gs(m)$ will stand for a whole class of sets. To simplify
calculations, we will also require $X_1$ to grow quickly.

We will define the $X_i$ by induction on $i$, together with
relations $\myll i$ on $\R$, which describe ``at least how fast''
the terms in $X_i$ must grow (but we will also use $\myll i$ for
comparing real numbers other than the members of $X_i$). Let us
write $X_i=\{x_{i1},x_{i2},\ldots,x_{im}\}$, where
$x_{i1}<x_{i2}<\cdots< x_{im}$.

We start by letting $x \myll 1 y$ to mean $K_1 x \le y$, where $K_1
= 2^d$. Then we choose $X_1$ so that $x_{11}=1$ and $x_{11}\myll 1
x_{12}\myll 1 \cdots\myll 1 x_{1m}$.

Having defined $X_{i-1}$ and $\myll {i-1}$, we set $K_i:= 2^d
x_{(i-1)m}$, we define $x\myll i y$ to mean $K_i x\le y$, and we
choose $X_i$ so that $x_{i1}=1$ and $x_{i1}\myll i x_{i2}\myll
i\cdots \myll i x_{im}$.

This construction develops further an idea from our earlier paper
\cite{BMN_fsl}. As we will explain, the intersections of convex sets
with the stretched grid can be approximated, up to a small error, by
sets that have a simple, essentially combinatorial description.

It is practically impossible to make a realistic drawing of the
stretched grid, but we can conveniently think about it using a
bijection with a uniform (equally spaced) grid. Namely, we define
the \emph{uniform grid} in the unit cube $\I^d$ by
\begin{equation*}
\Gu=\Gu(m):=\bigl\{\textstyle 0, {1\over m-1}, {2\over m-1}, \ldots,
{m-1\over m-1}\bigr\}^d.
\end{equation*}
Let $\B:=[1,x_{1m}]\times[1,x_{2m}]\times\cdots\times
[1,x_{dm}]$ be the bounding box of $\Gs$, and let $\pi\colon
\B\to \I^d$ be a bijection that maps $\Gs$ onto $\Gu$ and
preserves ordering in each coordinate (that is, we map points
of  $\Gs$ to the corresponding points of $\Gu$ and we squeeze
the ``elementary boxes'' of $\Gs$ onto the corresponding
elementary boxes of $\Gu$).

Fig.~\ref{fig_grid_hull} shows, for $d=2$, the image under $\pi$ of
two straight segments connecting grid points (left) and of a
``generic'' convex set (right). The image of the straight segment
$ab$, for example, first ascends almost vertically almost to the
level of $b$, and then it continues almost horizontally towards $b$.
This motivates the following notions.

\begin{figure*}
\centerline{\includegraphics{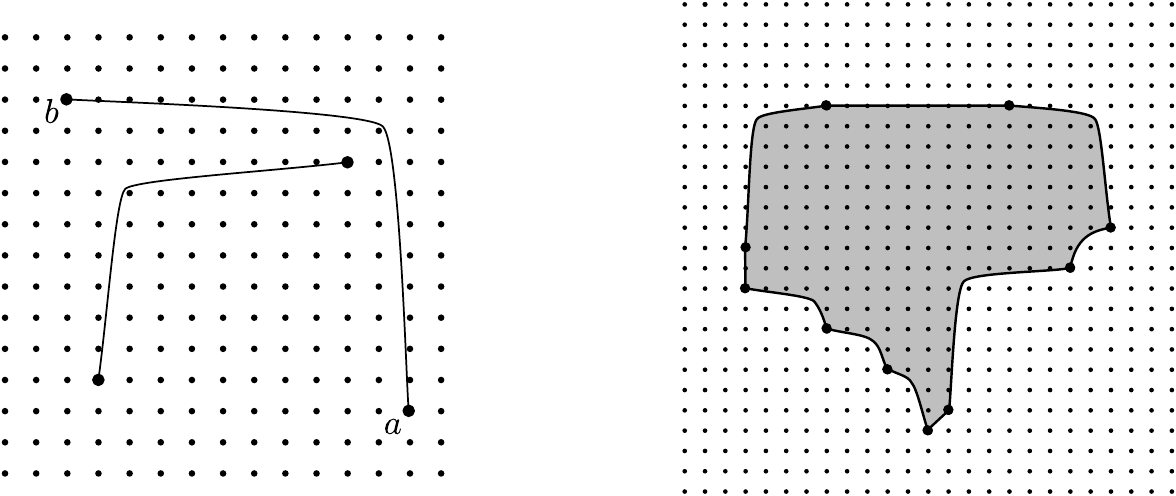}}
 \caption{\label{fig_grid_hull}
The bijection transforming the stretched grid to the uniform grid:
the images of two straight segments connecting grid points (left),
and the image of a convex set---the convex hull of the points marked
bold (right).}
\end{figure*}

\paragraph{Stair-convexity. }
First we define, for points $a=(a_1,a_2,\ldots,a_d)$ and
$b=(b_1,b_2,\ldots,b_d)\in\R^d$, the \emph{stair-path}
$\sigma(a,b)$. It is a polygonal path connecting $a$ and $b$ and
consisting of at most $d$ closed line segments, each parallel to one
of the coordinate axes. The definition goes by induction on $d$; for
$d=1$, $\sigma(a,b)$ is simply the segment $ab$. For $d\ge 2$, after
possibly interchanging $a$ and $b$, let us assume $a_d\le b_d$. We
set $a':=(a_1,a_2,\ldots,a_{d-1},b_d)$ and we let $\sigma(a,b)$ be
the union of the segment $aa'$ and of the stair-path $\sigma(a',b)$;
for the latter we use the recursive definition after ``forgetting''
the (common) last coordinate of $a'$ and $b$. See
Fig.~\ref{fig_stair-path} for examples.

\begin{figure*}
\centerline{\includegraphics{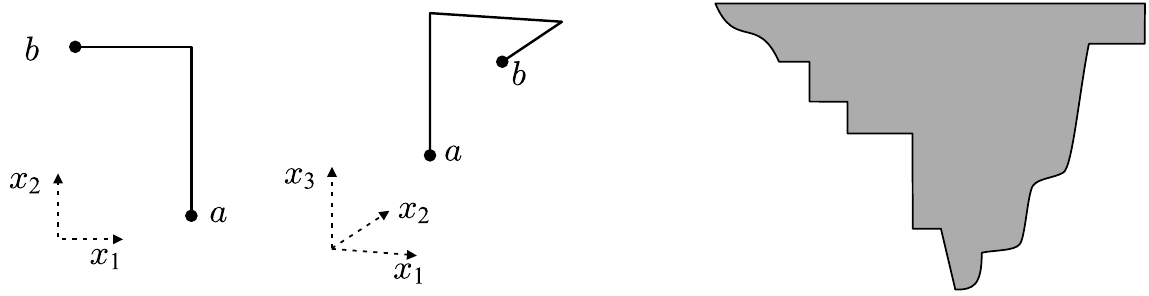}}
 \caption{\label{fig_stair-path} Examples of a stair-path in
the plane (left) and in $3$-space (center). An example of a
stair-convex set in the plane (right).}
\end{figure*}

Now we define a set $S\subseteq\R^d$ to be \emph{stair-convex} if
for every $a,b\in S$ we have $\sigma(a,b)\subseteq S$. See
Fig.~\ref{fig_stair-path} again.\footnote{Readers familiar with
abstract convex spaces might notice that a $d$-fold cone over the
one-element convex structure is almost the same as the family of
stair-convex subsets of $\I^d$. See Van de Vel
\cite[p.~32]{vandevel} for the definitions. We also note that any
line parallel to a coordinate axis intersects a stair-convex set in
a (possibly empty) segment. Sets with this latter property are
called \emph{rectilinearly convex}, \emph{orthoconvex}, or
\emph{separately convex} in various sources; however,
stair-convexity is a considerably stronger property. Another notion
somewhat resembling stair-convexity are the \emph{staircase
connected sets} studied by Magazanik and
Perles~\cite{staircase_connected}.}

Since the intersection of stair-convex sets is obviously
stair-convex, we can also define the \emph{stair-convex hull}
$\sconv(X)$ of a set $X\subseteq \R^d$ as the intersection of all
stair-convex sets containing~$X$.

As Fig.~\ref{fig_grid_hull} indicates, convex sets in the stretched
grid transform to ``almost'' stair-convex sets. We will now express
this connection formally.

\paragraph{\boldmath Epsilon-nets for stair-convex
sets and a transference lemma. }  Let us call a set $N\subseteq
\I^d$ an \emph{$\eps$-net for $\I^d$ with respect to
stair-convex sets}\footnote{In order to put this notion, as
well as weak $\eps$-nets introduced earlier, into a wider
context, we recall the following general definitions,
essentially due to Haussler and Welzl. Let $Y$ be a set, let
$\FF\subseteq 2^Y$ be a system of subsets of $Y$, and let $\mu$
be a finite measure on $Y$ such that all $F\in\FF$ are
measurable. A set $N\subseteq Y$ is a \emph{weak $\eps$-net for
$(Y,\FF)$ w.r.t.\ $\mu$} if $N\cap F\ne\emptyset$ for all
$F\in\FF$ with $\mu(F)\ge \eps\mu(Y)$. It is an $\eps$-net for
$(Y,\FF)$ w.r.t.\ $\mu$ if, moreover, $N$ is contained in the
support of~$\mu$. } if $N\cap S\ne\emptyset$ for every
stair-convex $S\subseteq \I^d$ with $\vol(S)\ge\eps$ (where
$\vol(\cdot)$ denotes the $d$-dimensional Lebesgue measure on
$\I^d$).

\begin{lemma}[Transference for weak $\eps$-nets]
 \label{l:transfer}\ \\[-5mm]
\begin{enumerate}
\item[\rm(i)]
Let $N$ be a weak $\eps$-net (w.r.t.\ convex sets) for the
$d$-dimensional stretched grid $\Gs = \Gs(m)$ of side $m$. Then the
set $\pi(N)\subseteq \I^d$ is an $\eps'$-net for $\I^d$ w.r.t.\
stair-convex sets with $\eps' \le \eps + O(|N|/m)$ (with the
constant of proportionality depending on $d$).
\item[\rm(ii)]
Let $N$ be an $\eps$-net for $\I^d$ w.r.t.\  stair-convex sets. Then
$\pi^{-1}(N)$ is a weak $\eps'$-net (w.r.t.\ convex sets) for
$\Gs(m)$ with $\eps' \le \eps + O(|N|/m)$, again with the constant
of proportionality depending on~$d$.
\end{enumerate}
\end{lemma}

The proof is based mainly on the next two lemmas, which will be
useful elsewhere as well. The first one is a local characterization
of the stair-convex hull.

Let $a = (a_1, \ldots, a_d)$ be a point in $\R^d$. We say that
another point $b = (b_1, \ldots, b_d) \in \R^d$ has \emph{type $0$
with respect to $a$} if $b_i \le a_i$ for every $i=1,2,\ldots,d$.
For $j\in\{1,2,\ldots,d\}$ we say that $b$ has \emph{type $j$ with
respect to $a$} if $b_j \ge a_j$ but $b_i\le a_i$ for all $i =
j+1,\ldots,d$. (It may happen that $b$ has more than one type with
respect to $a$, but only if some of the above inequalities are
equalities.)

\begin{lemma}\label{lemma_Carath_stair}
Let $X \subseteq \R^d$ be a point set, and let $x\in\R^d$ be a
point. Then $x\in\sconv(X)$ if and only if $X$ contains a point
of type $j$ with respect to $x$ for every $j=0,1,\ldots,d$.
\end{lemma}

The next lemma shows that convex hulls and stair-convex hulls almost
coincide in the stretched grid. Let us say that two points
$a=(a_1,\ldots,a_d)$ and $b= (b_1,\ldots,b_d)$ in $\B$ are \emph{far
apart} if, for every $i=1,2,\ldots,d$, we have either $a_i\myll i
b_i$ or $b_i\myll i a_i$. We also extend
this notion to sets; $P,Q\subseteq \R^d$ are far apart
if each $p\in P$ is far apart from each~$q\in Q$.

\begin{lemma}\label{lemma_conv_sconv}
Let $P$ and $Q$ be sets in $\B$ that are far apart.
Then $\sconv(P)\cap\sconv(Q)\neq
\emptyset$ if and only if $\conv(P)\cap\conv(Q)\neq\emptyset$.
\end{lemma}

In this paper we use Lemma~\ref{lemma_conv_sconv} only with $|Q| =
1$ (then it is a statement about membership of a point $q$ in
$\conv(P)$). We believe, however, that the above more general
version is interesting in its own right and potentially useful in
further applications, and thus worth expending some extra effort in
the proof. The lemma generalizes a result of \cite{BMN_fsl}, but the
proof method is different.

The proofs of Lemmas \ref{l:transfer}--\ref{lemma_conv_sconv}
are somewhat technical and can be skipped on first reading;
they appear in Section~\ref{sec_proofs}.\footnote{Also
see~\cite[sec.~2.1.3]{nivasch_phd} for a slightly different
proof of Lemma~\ref{lemma_conv_sconv} (for the case $|Q| = 1$)
from the one presented here.}

Theorem~\ref{t:wen-lwb} immediately follows from
Lemma~\ref{l:transfer}(i) and the next proposition:

\begin{prop}\label{p:rothlike}
Every $\frac 1r$-net for $\I^d$ w.r.t.\ stair-convex sets has at
least $\Omega(r\log^{d-1} r)$ points.
\end{prop}

The proof, which we present in Section~\ref{sec_stconv_lower_upper},
is strongly inspired by Roth's beautiful lower bound in discrepancy
theory \cite{r-id-54}; also see \cite{m-gd} for a presentation of
Roth's proof and a wider context.

As we also show in Section~\ref{sec_stconv_lower_upper}, the lower
bound in the proposition is actually tight (up to a constant
factor). This means, via Lemma~\ref{l:transfer}(ii), that the
stretched grid itself is not going to provide any stronger lower
bounds for weak $\eps$-nets than those proved here.

\paragraph{\boldmath
Weak $\eps$-nets for ``$1$-dimensional'' sets. } The smallest
possible size of weak $\eps$-nets has also been investigated for
special classes of sets \cite{CEGGSW,BC_sphere,MW,interval_chains}.

For us, two results of Alon, Kaplan, Nivasch, Sharir, and
Smorodinsky \cite{interval_chains} (see
also~\cite{nivasch_phd}) are particularly relevant. First,
improving on earlier results  by Chazelle et al.~\cite{CEGGSW},
they proved that for every planar finite set $X$ \emph{in
convex position} we have $f(X,r)= O(r\alpha(r))$, where
$\alpha$ denotes the inverse Ackermann function (we recall that
$f(X,r)$ is the smallest possible size of a weak $\frac1r$-net
for $X$). This, together with our Theorem~\ref{t:wen-lwb},
shows that the worst case for weak $\eps$-nets in the plane
does \emph{not} occur for sets in convex position.

Second, Alon et al.~\cite{interval_chains}, improving on
\cite{MW}, also showed that if $\gamma$ is a curve in $\R^d$
that intersects every hyperplane in at most $k$ points, where
$d$ and $k\ge d$ are considered constant, then every finite
$X\subset \gamma$ has weak $\frac1r$-nets of size almost linear
in $r$.\footnote{A curve in $\R^d$ that intersects every
hyperplane in at most $d$ points is called a \emph{convex
curve} in some sources, e.g. \v
Zivaljevi\'c~\cite[p.~314]{zival_handbook}.} We won't recall
the precise formulas, which are somewhat complicated; we just
state that the size can be bounded by $r\cdot
2^{C\alpha(r)^b}$, where $C$ and $b$ depend only on $d$
and~$k$.

We will show that for $d\ge 3$, point sets on a curve $\gamma$ as
above (with $k=d$) indeed require weak $\frac1r$-nets of size
superlinear in $r$ in the worst case, and the form of our lower
bound is actually similar to the just mentioned upper bounds, only
with smaller values of~$b$.

This time the point set is the \emph{diagonal} $\diag$ of the
$d$-dimensional stretched grid $\Gs$. That is, with
$\Gs(n)=X_1\times\cdots\times X_d$, where
$X_i=\{x_{i1},\ldots,x_{in}\}$, we set
$\diag(n):=\{(x_{1j},\ldots,x_{dj}):j=1,2,\ldots,n\}$ (this set
appeared already in \cite{BMN_fsl}, although there it was defined
slightly differently).

\begin{theorem}\label{t:alphas}
For $d\ge 3$ fixed, let us put $t:=\lfloor d/2\rfloor -1$, and
let us define a function $\beta_d$ by
\begin{equation*}
\beta_d(r) := \begin{cases}
\,\,\,\frac1{t!}\alpha(r)^t & \text{for $d$ even};\\
\,\,\,\frac1{t!}\alpha(r)^t \log_2\alpha(r) & \text{for $d$ odd}.
\end{cases}
\end{equation*}
\begin{enumerate}
\item[\rm(i)] {\rm (Lower bound) }
For every $r\ge 1$ there exists $n_0=n_0(r)$ such that for all $n\ge
n_0$
\begin{equation*}
f(\diag(n),r)\ge r\cdot 2^{(1-o(1))\beta_d(r)},
\end{equation*}
where $o(\cdot)$ refers to $r\to\infty$ and the $o(1)$ term
has the form $O(\alpha(r)^{-1})$ for $d$ even and
$O((\log_2\alpha(r))^{-1})$ for $d$ odd. (In particular,
for $d=3$ the lower bound is $\Omega(r\alpha(r))$.)
\item[\rm(ii)] {\rm (Upper bound) }
The lower bound from (i) is tight in the worst case up to the $o(1)$
term in the exponent. That is, $f(\diag(n),r)\le r\cdot
2^{(1+o(1))\beta_d(r)}$, with the same form of the $o(1)$  term as
in~(i).
\end{enumerate}
\end{theorem}

This theorem is proved in Section~\ref{sec_alpha}; the proof relies
essentially on tools from \cite{interval_chains}. In that section we
will also check that, with a suitable choice of the stretched grid,
the set $\diag$ is contained in a curve intersecting every
hyperplane at most $d$ times.

We find it quite fascinating that the bounds in the theorem are also
identical to the current best upper bounds for a seemingly unrelated
problem: the maximum possible length of \emph{Davenport--Schinzel
sequences}~\cite{DS_yo}.

\paragraph{``Thin'' sets of triangles. }
Let $X$ be an $n$-point set in the plane, and let $T$ be a family of
$t$ triangles with vertices at the points of~$X$. B\'ar\'any,
F\"uredi, and Lov\'asz \cite{bfl-nhp-90} were the first to prove a
statement of the following kind: If $T$ has ``many'' triangles, then
there is a point contained in a ``considerable number'' of triangles
of $T$. (This kind of statement is called a \emph{second selection
lemma} in \cite{matou_DG}. B\'ar\'any et al.\ used it in their proof
of the first nontrivial upper bound in the so-called \emph{$k$-set
problem} in dimension~3, and their work inspired many further
exciting results such as the \emph{colored Tverberg theorem}; see,
e.g., \cite{matou_DG} for background.) The current best quantitative
version is this: There exists a point contained in at least
$\Omega\bigl( t^3 / (n^6 \log^2 n) \bigr)$ triangles of~$T$ (Nivasch
and Sharir \cite{NS}, fixing a proof of Eppstein \cite{eppstein}).

It is not hard to see that this lower bound cannot be improved
beyond $O(t^2/n^3)$. Indeed Eppstein \cite{eppstein} showed that for
\emph{every} $n$-point set $X\subset \R^2$ and for all $t$ between
$n^2$ and $n\choose 3$ there is a set of $t$ triangles with vertices
in $X$ such that no point lies in more than $O(t^2/n^3)$ triangles
of~$T$.

Here we provide the first (slight) improvement of this easy bound,
again using the stretched grid.

\begin{theorem}\label{t:2ndsel}
Let $n=m^2$. Then for all $t$ ranging from $n^{2.5}\log n$ to
$n\choose 3$ there exists a set of $t$ triangles on the stretched
grid $\Gs(m)$ such that no point lies in more than
\begin{equation*}
O\Bigl(\frac{t^2}{n^3\log (n^3/t)}\Bigr)
\end{equation*}
triangles of~$T$. (In particular, if $t<n^{3-\delta}$ for some
constant $\delta>0$, then the bound is $O(t^2/(n^3\log n))$.)
\end{theorem}

This theorem is proved in Section~\ref{sec_SSL}.

\section{Epsilon-nets with respect to stair-convex sets}\label{sec_stconv_lower_upper}

Here we prove Proposition~\ref{p:rothlike}, stating that every
$\frac 1r$-net for $\I^d$ w.r.t.\ stair-convex sets has $\Omega(r
\log^{d-1} r)$ points. Thus, for an arbitrary set $N\subseteq \I^d$
of $n$ points, it suffices to exhibit a stair-convex set $S\subseteq
\I^d$ of volume at least $\Omega((\log^{d-1} n)/n)$ that avoids~$N$.

We will produce such an $S$ as a union of suitable axis-parallel
boxes.

Let $k=\Theta(\log n)$ be the integer with $2^{d+1}n\le 2^k< 2^{d+2}
n$, and let us call every integer vector $t=(t_1,t_2,\ldots,t_d)$
with $t_i\ge 1$ for all $i$ and with $t_1+t_2+\cdots+t_d=k$ a
\emph{box type}. For later use we record that the number $T$ of box
types is ${k-1\choose d-1}=\Omega(k^{d-1})$.

Let $V:=[\frac12,1]^d$ be the ``upper right part'' of the cube
$\I^d$. For a box type $t$ and a point $p\in V$, we define the
\emph{normal box of type $t$ anchored at $p$} as
\begin{equation*}
B_t(p) := [p_1-2^{-t_1},p_1]\times
          [p_2-2^{-t_2},p_2]\times\cdots\times
          [p_d-2^{-t_d},p_d].
\end{equation*}
Since each side of $B_t(p)$ is at most $\frac 12$, each normal box
is contained in $\I^d$.

The volume of each normal box is $2^{-k}\le 1/(2^{d+1}n)$. Let us
call a normal box $B_t(p)$ \emph{empty} if $B_t(p)\cap N=\emptyset$.
We will show that for every box type $t$ and for $p\in V$ chosen
uniformly at random, we have
\begin{equation}\label{e:emptynormal}
\Pr[B_t(p)\mbox{ is empty}]\ge{\textstyle \frac 12}.
\end{equation}
Indeed, for every point $x\in \I^d$ we have $\vol\{p\in V: x\in
B_t(p)\}\le 2^{-k}$, which in probabilistic terms means $\Pr[x\in
B_t(p)]\le 2^{-k}/\vol(V)=2^{-k+d}\le \frac 1{2n}$, and
(\ref{e:emptynormal}) follows by the union bound.

Now we define the \emph{fan} $\FF(p)$ of a point $p\in V$ as the set
consisting of the normal boxes $B_t(p)$ for all the $T$ possible box
types $t$ (see Fig.~\ref{fig_fan} left). By (\ref{e:emptynormal}) we
get that for a random $p\in V$ the expected number of empty boxes in
the fan of $p$ is at least $T/2$.

\begin{figure}
\centerline{\includegraphics{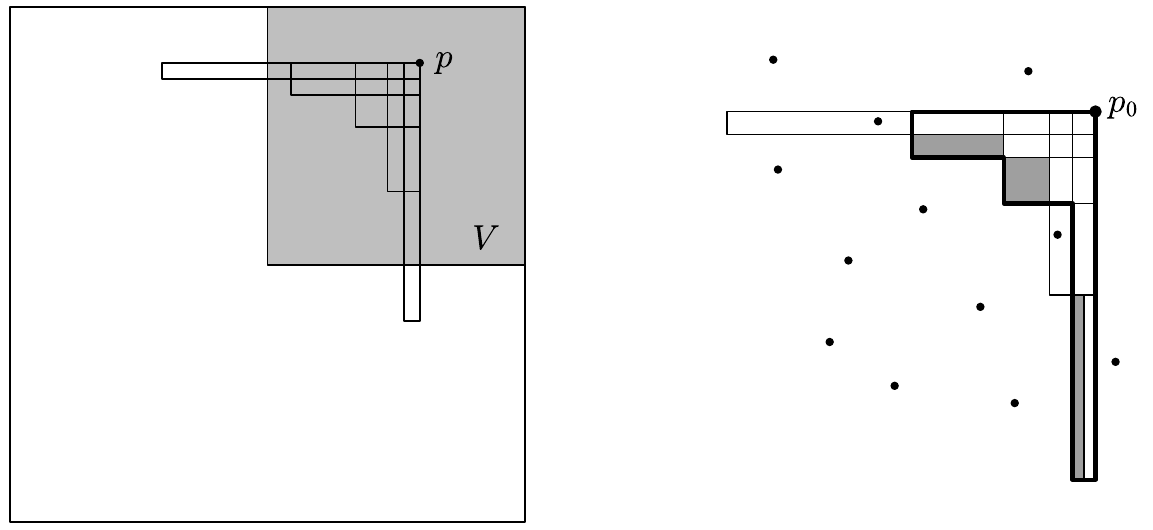}}
 \caption{\label{fig_fan}
The fan $\FF(p)$ (left); the stair-convex set $S$ made
of the empty boxes of $\FF(p_0)$ and the lower subboxes
witnessing the volume of~$S$ (right).}
\end{figure}

Thus, there exists a particular point $p_0\in V$ such that
$\FF(p_0)$ has at least $T/2$ empty boxes. We define $S$ as the
union of these empty boxes. Then $S\cap N=\emptyset$, $S$ is clearly
stair-convex, and it remains to bound from below the volume of~$S$.

For an axis-parallel box
$B=[a_1,a_1+s_1]\times\cdots\times[a_d,a_d+s_d]$ we define the
\emph{lower subbox} $B':= [a_1,a_1+\frac12 s_1]\times\cdots\times
[a_d,a_d+\frac12 s_d]$. We observe that if $B_{t_1}(p)$ and
$B_{t_2}(p)$ are two normal boxes of different types anchored at the
same point, then their lower subboxes are disjoint. Hence, $\vol(S)$
is at least the sum of volumes of the lower subboxes of $T/2$ normal
boxes, and so $\vol(S)\ge\frac T2 2^{-d} 2^{-k}=
\Omega((\log^{d-1}n)/n)$. Proposition~\ref{p:rothlike} is proved.
\proofend

\medskip

Now we show that Proposition~\ref{p:rothlike} is asymptotically
tight; namely, that for every $r\ge 1$ there exists a set $N \subset
\I^d$, $|N| = O(r \log^{d-1} r)$, intersecting every stair-convex
$S\subseteq \I^d$ with $\vol(S) \ge \frac 1r$.

We begin with the following fact: {\em For every $s\ge 1$ there
exists a set $N\subset \I^d$ of size $O(s)$ intersecting every
axis-parallel box $B\subseteq \I^d$ with $\vol(B)\ge\frac1s$. }
Indeed, the Van der Corput set in the plane and the
Halton--Hammersley sets in dimension $d$ have this property, as well
as many other constructions of low-discrepancy sets (Faure sets,
digital nets of Sobol, Niederreiter and others, etc.); see, e.g.,
\cite{m-gd}.

Given $r\ge 1$, we now set $s:= Cr\log^{d-1}r$ for a sufficiently
large constant $C$, and we let $N$ be a set as in the just mentioned
fact.  We claim that $N$ is the desired $1\over r$-net for $\I^d$
w.r.t.\ stair-convex sets. This follows from the next lemma.

\begin{lemma}
Let $S\subseteq \I^d$ be a stair-convex set that contains no
axis-parallel box of volume larger than $v$, $0<v\le1/e$ (here,
$e=2.71828\ldots$). Then $\vol(S)\le ev \ln^{d-1} {1\over v}$.
\end{lemma}

\begin{proof}
We proceed by induction on $d$. The base case $d=1$ is trivial, so
we assume $d\ge 2$.

Without loss of generality we can assume $S$ intersects the ``upper
facet" of $\I^d$ (the facet of $\I^d$ with last coordinate equal to
$1$).

For $z\in [0,1]$ let $h=h(1-z)$ denote the ``horizontal'' hyperplane
$\{x\in \R^d : x_d=1-z\}$. Let $S' := S\cap h$, and let $B$ be an
axis-parallel box of maximum $(d-1)$-dimensional volume in $S'$. We
have $\vol_{d-1}(B) \le \frac v  z$, for otherwise, $B$ could be
extended upwards into a box of $d$-dimensional volume larger
than~$v$.

Since $S'$ is stair-convex, for $z\ge ev$ the inductive assumption
gives $\vol(S') \le {ev\over z} \ln^{d-2} {z\over v}$. We also have
$\vol(S') \le 1$. So for $v\le1/e$ we have
\begin{eqnarray*}
\vol(S) &\le &\int_0^{ev} {\rm d}z +\int_{ev}^{1} {ev \over z}
\ln^{d-2} {z\over v} \,{\rm d}z
= ev+ {ev\over d-1} \left(\ln^{d-1} {1\over v}-1\right)\\
&\le& ev+ev\left(\ln^{d-1} {1\over v}-1\right) = ev \ln^{d-1}
{1\over v}.
\end{eqnarray*}
This finishes the induction step.
\end{proof}

\section{The upper bound for the second selection lemma}\label{sec_SSL}

\begin{proof}[Proof of Theorem~\ref{t:2ndsel}. ]
We consider $n=m^2$ and the planar stretched grid $\Gs(m)$. Let us
now write $\Gs(m)=\{x_1,\ldots,x_m\}\times \{y_1,\ldots,y_m\}$. We
want to define a set $T$ of $t$ triangles with vertices in $\Gs(m)$
that is ``thin'', i.e., no point is contained in too many triangles.

Let $\rho\in (0,1]$ be a parameter, which we will later determine in
terms of $n$ and $t$.

Let $p_1 = (x_{i_1}, y_{j_1})$, $p_2 = (x_{i_2}, y_{j_2})$, $p_3 =
(x_{i_3}, y_{j_3})$ be three distinct points of $\Gs(m)$. Let us
call the triangle $\Delta = p_1p_2p_3$  \emph{increasing} if $i_1 <
i_2 <i_3$ and $j_1 < j_2 < j_3$. Let us define the \emph{horizontal
dimensions} of $\Delta$ as $h_{12}:=i_2 - i_1$ and $h_{23}:=
i_3-i_2$, and the \emph{vertical dimensions} as $v_{12}:=j_2-j_1$
and $v_{23}:=j_3-j_2$.

We define $T$ as the set of all increasing triangles $\Delta$
as above that satisfy
\begin{equation*}
\textstyle
\frac 13m \le i_2,j_2 \le {2\over 3}m;\ \ \ \ \ \ \ \
h_{12},h_{23},v_{12},v_{23}\le \frac13m;\ \ \ \ \ \ \
h_{12}v_{23} \le \rho n.
\end{equation*}
The last condition may look mysterious but it will be explained
soon. However, first we bound $|T|$ from below, which is routine.

An increasing triangle $\Delta$ is determined by $p_2$ and by its
horizontal and vertical dimensions. Each of $i_2,j_2,h_{23},v_{12}$
can be chosen independently in $\frac m3$ ways. The pair
$(h_{12},v_{23})$ can then be chosen, independent of the previous
choices, as a lattice point lying in the square $[0,\frac m3]^2$ and
below the hyperbola $xy=\rho n$, and one can easily calculate (by
integration, say) that the number of choices is of order $\rho n
\log\frac1\rho$. Thus $|T|=\Omega(n^3 \rho\log\frac1\rho)$, and thus
for $\rho:=C t/(n^3\log(n^3/t))$ with a sufficiently large constant
$C$ we obtain $|T|\ge t$ as needed. (Actually, the above calculation
of integer points under the hyperbola is valid only if $\rho$ is not
too small compared to $m$, but the assumptions of the theorem and
our choice of $\rho$ guarantee $\rho=\Omega(\frac1m)$.)

Let us fix an arbitrary point $q$ in the plane. It remains to bound
from above the number of triangles $\Delta\in T$ containing $q$. To
this end, we partition the triangles in $T$ into classes according
to their horizontal and vertical dimensions; let
$T(h_{12},h_{23},v_{12},v_{23})$ be one of these classes. The total
number of triangles in such an equivalence class equals the number
of choices of $p_2$, so it is  $\Theta(n)$. We want to show that
only $O(\rho n)$ of them contain~$q$.

We use Lemma~\ref{lemma_conv_sconv} with $P = \{p_1, p_2, p_3\}$ and
$Q = \{q\}$. Then, $q\in\Delta$ may hold only if
$q\in\sconv\{p_1,p_2,p_3\}$ or if $q$ is not far apart from at least
one of $p_1,p_2,p_3$.

If, say, $p_2$ is not far apart from $q$, then its position is
restricted to two rows or two columns of the grid, and similarly for
$p_1$ and $p_3$. Thus, there are only $O(m)$ choices for $\Delta$.

\begin{figure}
\centerline{\includegraphics{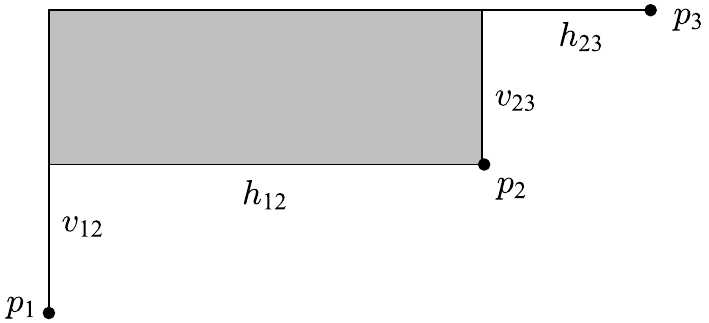}}
 \caption{\label{fig_incr_triangle}
 The  stair-convex hull of
the vertex set of a triangle in $T(h_{12},h_{23},v_{12},v_{23})$.}
\end{figure}

It remains to deal with the case $q\in\sconv\{p_1,p_2,p_3\}$.
The stair-convex hull of the vertex set of a triangle
$\Delta\in T(h_{12},h_{23},v_{12},v_{23})$ is depicted in
Fig.~\ref{fig_incr_triangle} (the picture actually shows the
image under $\pi$ in the uniform grid). It contains
$h_{12}v_{23}+O(m)\le \rho n+O(m)$ grid points, and thus there
are at most $\rho n+O(m)=O(\rho n)$ placements of $p_2$ such
that the stair-convex hull of the vertex set contains $q$.

So in every equivalence class of the triangles of $T$ only an
$O(\rho)$ fraction of triangles contain $q$. Thus $q$ lies in no
more than $O(\rho|T|)=O\bigl(t^2/(n^3\log (n^3/t))\bigr)$ triangles
of $T$ as claimed.
\end{proof}

\begin{remark}
A related problem calls for constructing a set of $t$ triangles
spanned by $n$ points in $\R^3$, such that no \emph{line} in $\R^3$
stabs too many triangles. The above upper bound does \emph{not}
generalize to this latter problem. This fact gives more weight to
our conjecture \cite{BMN_fsl} that the latter, three-dimensional
problem has a larger bound than the planar problem.
\end{remark}

\paragraph{The first selection lemma and
generalizations. }\label{sec_FSL} In \cite{BMN_fsl} we gave an
improved upper bound for the so-called \emph{first selection lemma},
by constructing an $n$-point set $X$ in $\R^d$ such that no point in
$\R^d$ is contained in more than $\bigl(\frac n {d+1} \bigr)^{d+1} +
O(n^d)$ of the $d$-dimensional simplices spanned by $X$. The
construction was precisely the ``main diagonal" $\diag$ of the
stretched grid~$\Gs$.

Now this can be regarded as a special case of the following result:

\begin{prop}\label{prop_FSL_upper_gral}
Let $X\subset\R^d$ be an $n$-point subset of $\Gs(m)$ for some $m$
such that every hyperplane perpendicular to a coordinate axis
contains only $o(n)$ points of $X$. (In particular, $X$ can be $\Gs$
itself.) Then no point $q\in\R^d$ is contained in more than
$(1+o(1))\bigl( \frac n{d+1} \bigr)^{d+1}$ of the $d$-simplices with
vertices in~$X$.
\end{prop}

\begin{proof}
This follows immediately from the arithmetic-geometric mean
inequality since, by Lemmas~\ref{lemma_Carath_stair}
and \ref{lemma_conv_sconv}, every simplex
that contains $q$ (except for at most $o(n^{d+1})$ simplices that
have a vertex not lying far apart from  $q$) must have one vertex of
each type with respect to~$q$.
\end{proof}

On a related topic, our calculations show that in dimension
$3$, if we let $X:=\Gs(\sqrt[3]{n})$, then no line in $\R^3$
intersects more than $n^3/25 + o(n^3)$ triangles spanned by
$X$. This proves tightness of another result in \cite{BMN_fsl}
(assuming our calculations are correct). Unfortunately, the
calculations, although essentially straightforward, are rather
tedious and do not seem to generalize easily. (We would like to
find, for general $d$, $j$, and $k$, the maximum number of
$j$-simplices spanned by points of the $d$-dimensional
stretched grid that can be stabbed by a $k$-flat in $\R^d$.)

\section{The diagonal of the stretched grid}\label{sec_alpha}

Here we prove our results on the diagonal $\diag = \diag(n)$ of the
stretched grid $\Gs(n)$. We start by showing that, if $\Gs$ is
defined appropriately, then $\diag$ lies on a curve that intersects
every hyperplane in at most $d$ points.

Indeed, if each element $x_{ij}$ of each $X_i$ in the definition of
$\Gs$ is chosen \emph{minimally}, then we have $x_{ij} = K_i^{j-1}$,
and so
\begin{equation*}
\diag =  \bigl\{ (K_1^t, \ldots, K_d^t) : t = 0, 1, \ldots, n-1
\bigr\}.
\end{equation*}
Thus, $\diag$ is a subset of the curve
\begin{equation*}
\gamma = \bigl\{ (K_1^t, \ldots, K_d^t) : t\in \R \bigr\}.
\end{equation*}

\begin{lemma}
Let $\gamma\subset \R^d$ be a curve of the form
\begin{equation*}
\gamma = \bigl\{ (c_1^t, \ldots, c_d^t) : t\in \R \bigr\},
\end{equation*}
for some positive constants $c_1, \ldots, c_d$. Then every
hyperplane in $\R^d$ intersects $\gamma$ at most $d$ times.
\end{lemma}

\begin{proof}
The claim is equivalent to showing that the function
\begin{equation*}
f(t)=\alpha_1 c_1^t + \cdots + \alpha_d c_d^t + \alpha_{d+1}
\end{equation*}
has at most $d$ zeros for any choice of parameters $\alpha_1,
\ldots, \alpha_{d+1}$. Letting $\beta_i = \alpha_i\ln c_i$, it
suffices to show that
\begin{equation*}
f'(t) = \beta_1 c_1^t + \beta_2 c_2^t + \cdots + \beta_d c_d^t =
c_1^t \bigl( \beta_1 + \beta_2 (c_2/c_1)^t + \cdots + \beta_d
(c_d/c_1)^t \bigr)
\end{equation*}
has at most $d-1$ zeros. But $c_1^t$ never equals zero, so the claim
follows by induction.
\end{proof}

Next, we prove Theorem~\ref{t:alphas}, the lower and upper bounds
for the size of weak $\frac1r$-nets for the diagonal $\diag$ of the
stretched grid. We reduce the problem to results of Alon et
al.~\cite{interval_chains} concerning the problem of \emph{stabbing
interval chains}.

\subsection{The Ackermann function and its inverse}

We introduce the Ackermann function and its inverse following
\cite{DS_yo}:

The \emph{Ackermann hierarchy} is a sequence of functions $A_k(n)$,
for $k \ge 1$ and $n\ge 0$, where $A_1(n) = 2n$, and for $k\ge 2$ we
let $A_k(n) = A_{k-1}^{(n)}(1)$. (Here $f^{(n)}$ denotes the
$n$-fold composition of $f$.) The definition of $A_k(n)$ for $k\ge
2$ can also be written recursively: $A_k(0) = 1$, and $A_k(n)
=A_{k-1}{\bigl( A_k(n-1) \bigr)}$ for $n\ge 1$. We have $A_2(n) =
2^n$, and $A_3(n) = 2^{2^{\cdots^2}}$ is a tower of $n$ twos.

We have $A_k(1) = 2$ and $A_k(2) = 4$, but $A_k(3)$ already grows
very rapidly with $k$. We define the \emph{Ackermann function} as
$A(n) = A_n(3)$. Thus, $A(n) = 6, 8, 16, 65536, \ldots$ for $n = 1,
2, 3, \ldots$.

We then define the slow-growing inverses of these rapidly-growing
functions as $ \alpha_k(x) = \min\{ n : A_k(n) \ge x \}$ and
$\alpha(x) = \min\{ n : A(n) \ge x\}$ for all real $x\ge 0$.

Alternatively, and equivalently, we can define these inverse
functions directly: We define the \emph{inverse Ackermann
hierarchy} by letting $\alpha_1(x) = \lceil x/2 \rceil$ and,
for $k\ge 2$, defining $\alpha_k(x)$ recursively by
\begin{equation}\label{eq_rec_alpha_k}
\alpha_k(x) =
\begin{cases}
0, & \text{if $x\le 1$};\\
1+ \alpha_k{\bigl( \alpha_{k-1}(x) \bigr)}, & \text{otherwise}.
\end{cases}
\end{equation}
In other words, for each $k\ge 2$, $\alpha_k(x)$ denotes the
number of times we must apply $\alpha_{k-1}$, starting from
$x$, until we reach a value not larger than $1$. Thus,
$\alpha_2(x) = \lceil \log_2 x \rceil$, and $\alpha_3(x) =
\log^*x$. Finally, we define the \emph{inverse Ackermann
function} by $\alpha(x) = \min{\{ k : \alpha_k(x) \le 3 \}}$.

Note that, by definition, we have $\alpha_{\alpha(x)}(x) \le 3$
and $\alpha_{\alpha(x)-1}(x) \ge 4$. Furthermore,
$\alpha_{\alpha(x) - 2}(x) \ge 5$ (since $\alpha_{k-1}(x) >
\alpha_k(x)$ whenever $\alpha_{k-1}(x) \ge 4$). We now show
that $\alpha_{\alpha(x) - 3}(x)$ grows to infinity with $x$,
and in fact it does so much faster than $\alpha(x)$ (though
obviously slower than $\alpha_k(x)$ for every fixed $k$):

\begin{lemma}\label{lemma_alpha_alpha_3}
Let $x$ be large enough so that $\alpha(x) \ge 4$. Then,
\begin{equation*}
\alpha_{\alpha(x) - 3}(x) > A(\alpha(x) - 2).
\end{equation*}
\end{lemma}

\begin{proof}
As noted above, we have $\alpha_{\alpha(x) - 2}(x) \ge 5$.
Thus, by (\ref{eq_rec_alpha_k}),
\begin{equation*}
5 \le \alpha_{\alpha(x) - 2}(x) = 1 + \alpha_{\alpha(x) - 2}{\bigl( \alpha_{\alpha(x) - 3}(x) \bigr)},
\end{equation*}
so $\alpha_{\alpha(x) - 2}{\bigl( \alpha_{\alpha(x) - 3}(x)
\bigr)} \ge 4$. But $\alpha_k(y) \ge 4$ implies $k \le
\alpha(y)-1$, so in our case,
\begin{equation*}
\alpha(x) - 1\le \alpha{\bigl( \alpha_{\alpha(x) - 3}(x) \bigr)}.
\end{equation*}
Finally, $n \le \alpha(y)$ implies $y > A(n-1)$, and the lemma
follows.
\end{proof}

The fact that $\alpha_{\alpha(x) - 3}(x) \to\infty$ will be
used below.

\subsection{Stabbing interval chains}

We now recall the problem of stabbing interval chains and the
bounds obtained in \cite{interval_chains, nivasch_phd}.

Let $[i,j]$ denote the interval of integers $\{i,i+1,\ldots,j\}$. An
\emph{interval chain} of size $k$ (also called a \emph{$k$-chain})
is a sequence of $k$ consecutive, disjoint, nonempty intervals
\begin{equation*}
C = I_1I_2\cdots I_k = [a_1,a_2][a_2+1,a_3]\cdots[a_k+1,a_{k+1}],
\end{equation*}
where $a_1 \le a_2 < a_3 < \cdots < a_{k+1}$. We say that a
$j$-tuple of integers $(p_1, \ldots, p_j)$ \emph{stabs} an interval
chain $C$ if each $p_i$ lies in a different interval of $C$.

The problem is to stab, with as few $j$-tuples as possible, all
interval chains of size $k$ that lie within a given range $[1,n]$.
We let $\z^{(j)}_k(n)$ denote the minimum size of a collection $Z$
of $j$-tuples that stab all $k$-chains that lie in $[1,n]$.

Alon et al.~showed in \cite{interval_chains, nivasch_phd} that,
for every fixed $j\ge 3$, once $k$ is large enough,
$\z^{(j)}_k(n)$ has near-linear lower and upper bounds roughly
of the form $n \alpha_m(n)$, where $m$ grows with $k$.
Specifically:

\begin{theorem}[Interval-chain lower bounds \cite{nivasch_phd}]\label{thm_z_j_lower}
Let $j\ge 3$ be fixed, and let $t = \lfloor j/2 \rfloor-1$. Then
there exists a function $Q_j(m)$ of the form
\begin{equation*}
Q_3(m) = 2m+1, \qquad Q_4(m) = \Omega{\bigl( 2^m \bigr)},
\end{equation*}
and, in general,
\begin{equation*}
Q_j(m) \ge
\begin{cases}
2^{(1/t!) m^t - O(m^{t-1})}, & \text{$j$ even};\\
2^{(1/t!) m^t \log_2 m - O(m^t)}, & \text{$j$ odd};
\end{cases}
\end{equation*}
such that, for all $m\ge 3$, if $k\le Q_j(m)$ then
\begin{equation}\label{eq_thm_z_lower}
\z^{(j)}_{k}(n) \ge c_j n \alpha_m(n) - c'_j n \qquad \text{for all $n$},
\end{equation}
for some constants $c_j$ and $c'_j$ that depend only on
$j$.\footnote{These lower bounds were stated
in~\cite{interval_chains} in a somewhat weaker form; we need
this stronger formulation from~\cite{nivasch_phd} in our
application below.}
\end{theorem}

\begin{theorem}[Interval-chain upper bounds
\cite{interval_chains,nivasch_phd}]\label{thm_z_j_upper}Let
$j\ge 3$ be fixed, and let $t = \lfloor j / 2\rfloor-1$. Then
there exists a function $P_j(m)$ of the form
\begin{equation*}
P_3(m) = 2m, \qquad P_4(m) = O{\bigl( 2^m \bigr)},
\end{equation*}
and, in general,
\begin{equation*}
P_j(m) \le
\begin{cases}
2^{(1/t!) m^t + O(m^{t-1})}, & \text{$j$ even};\\
2^{(1/t!) m^t \log_2 m + O(m^t)}, & \text{$j$ odd};
\end{cases}
\end{equation*}
such that, for all $m\ge 3$, if $k\ge P_j(m)$ then
\begin{equation}\label{eq_thm_z_upper}
\z^{(j)}_{k}(n) \le c''_j n \alpha_m(n) \qquad \text{for all $n$},
\end{equation}
for some constants $c''_j$ that depend only on $j$.
\end{theorem}

\subsection{Proof of the lower bounds}

\begin{lemma}\label{lemma_ell_to_z_lower}
Given $r>1$, let $N$ be a weak $1\over r$-net for $\diag =
\diag(n)$, for $n = n(r)$ large enough. Let $\ell = |N|$. Then
$\ell$ must satisfy
\begin{equation*}
\ell \ge \z^{(d)}_{4d\ell / r} (\ell).
\end{equation*}
\end{lemma}

\begin{proof}
For each point $x\in N$ and each coordinate $1\le j\le d$, mark as
``bad" the two points of $\diag$ that surround $x$ when the points
are projected into the $j$-th coordinate. Thus, at most $2d\ell$
points of $\diag$ are marked ``bad".

Partition $\diag$ into $4d\ell$ contiguous blocks of size
$n/(4d\ell)$ each (we can safely ignore the rounding to integers if
$n$ is large enough). Then there are $2d\ell$ blocks $B_1, \ldots,
B_{2d\ell}$ which are ``good", in the sense that they do not contain
any bad points. Place $2d\ell-1$ abstract ``separators" $Y_1,
\ldots, Y_{2d\ell-1}$ between these blocks, such that $Y_i$ lies
between $B_i$ and $B_{i+1}$.

Let $k = 4d\ell/r$. There is a natural one-to-one correspondence
between sets $\mathcal B$ of $k$ good blocks, and $(k-1)$-chains
$\mathcal B'$ on the separators. Namely, for every $i_1 < i_2 <
\cdots < i_k$ we map
\begin{equation*}
\mathcal B = \{ B_{i_1}, \ldots, B_{i_k}\} \leftrightarrow \
\mathcal B' = [Y_{i_1}, Y_{i_2-1}][Y_{i_2},Y_{i_3}-1]\cdots
[Y_{i_{k-1}},Y_{i_k-1}],
\end{equation*}
where the notation $[Y_a, Y_b]$ means $\{ Y_a, Y_{a+1}, \ldots,
Y_b\}$.

Let $\mathcal B = \{ B_{i_1}, \ldots, B_{i_k}\}$ be an
arbitrary such set. Let $\diag' = B_{i_1} \cup \cdots \cup
B_{i_k} \subseteq \diag$. Since $|\diag'| = n/r$ and $N$ is a
weak ${1\over r}$-net for $\diag$, it follows that
$\conv(\diag')$ must contain some point $x \in N$. By
Carath\'eodory's theorem, $x$ is contained in the convex hull
of some $d+1$ points of $\diag'$; let these points be $q_0,
\ldots, q_d$ from left to right.

Recall that for each coordinate $1\le j\le d$, the projection of $x$
into the $j$-th coordinate falls between two bad points of $\diag$.
Therefore, all the projections of $x$ fall \emph{between} good
blocks, and so we can associate with $x$ a $d$-tuple of separators
\begin{equation*}
x' = (Y_{a_1}, \ldots, Y_{a_d}).
\end{equation*}

Furthermore, none of the points $q_0, \ldots, q_d$ are bad, and
therefore they are \emph{far apart} from $x$ in each coordinate.
Therefore, Lemmas~\ref{lemma_Carath_stair} and
\ref{lemma_conv_sconv} apply, and so the $j$-th coordinate of $x$
must lie between the $j$-th coordinates of $q_{j-1}$ and $q_j$, for
every $j=1,2,\ldots,d$.

It follows that $q_0, \ldots, q_d$ belong to $d+1$ \emph{distinct}
blocks $B'_0, \ldots, B'_d$ of $\mathcal B$, and furthermore, their
relative order with the separators of $x'$ is
\begin{equation*}
B'_0, Y_{a_1}, B'_1, Y_{a_2}, \ldots, Y_{a_d}, B'_d.
\end{equation*}
In other words, the $d$-tuple $x'$ stabs the $(k-1)$-chain $\mathcal
B'$.

Thus, $N$ must have enough points to stab all $(k-1)$-chains (and so
all $k$-chains) with $d$-tuples in the range $[1,2d\ell-1] \supseteq
[1,\ell]$. Therefore,
\begin{equation*}
\ell = |N| \ge \z^d_{k}(\ell) = \z^d_{4d\ell/r}(\ell). \qedhere
\end{equation*}
\end{proof}

\begin{corollary}\label{cor_ell}
The quantity $\ell$ of Lemma \ref{lemma_ell_to_z_lower} must
satisfy
\begin{equation*}
\ell = \Omega{\bigl( r\cdot Q_d(\alpha(r) - 3) \bigr)},
\end{equation*}
for the function $Q_d$ of Theorem \ref{thm_z_j_lower}.
\end{corollary}

Note that Corollary \ref{cor_ell} implies Theorem~\ref{t:alphas}(i).

\begin{proof}
Suppose for a contradiction that $\ell \le {1\over 4d} r\cdot
Q_d(\alpha(r) - 3)$. Then $4d\ell / r \le Q_d(\alpha(r) - 3)$,
so by Lemma~\ref{lemma_ell_to_z_lower},
Theorem~\ref{thm_z_j_lower}, and
Lemma~\ref{lemma_alpha_alpha_3},
\begin{equation*}
\ell \ge \z^{(d)}_{4d\ell/r}(\ell) \ge \z^{(d)}_{Q_d(\alpha(r) - 3)}(\ell)
\ge c_d \ell \alpha_{\alpha(r) - 3}(\ell) - c'_d\ell
\ge c_d \ell \alpha_{\alpha(\ell) - 3}(\ell) - c'_d\ell = \omega(\ell).
\end{equation*}
(We have $\ell \ge r$ since every weak $1\over r$-net must
trivially have at least $r$ points.) This is a contradiction
for all large enough $\ell$, and so for all large enough $r$.
\end{proof}

\subsection{Proof of the upper bounds }

\begin{lemma}\label{lemma_to_z_upper}
The set $\diag = \diag(n)$ has a weak $1\over r$-net of size at most
$\z^{(d)}_{\ell/r - 1}(\ell)$, where $\ell$ is a free parameter.
\end{lemma}

\begin{proof}
Given $\ell$, partition $\diag$ into $\ell$ equal-sized blocks $B_1,
B_2, \ldots, B_{\ell}$ of consecutive points, leaving a pair of
adjacent points $Y_i = \{y_i, y'_i\}$ between every two consecutive
blocks $B_i$, $B_{i+1}$. We call the pairs of points $Y_1, \ldots,
Y_{\ell-1}$ ``separators". We assume $\ell$ is much smaller than
$n$, so the size of each block $B_i$ can be approximated by
$n/\ell$.

Consider a set $\diag' \subset \diag$ of size at least $n/r$.
$\diag'$ must contain a set $Q = \{q_1, \ldots, q_k\}$ of $k =
\ell/r$ points lying on $k$ different blocks $\mathcal B = \{
B_{i_1}, \ldots, B_{i_k} \}$, $i_1 < \cdots < i_k$. These blocks
define a $(k-1)$-chain of separators
\begin{equation*}
\mathcal B' = [Y_{i_1}, Y_{i_2 - 1}] [Y_{i_2}, Y_{i_3 - 1}] \cdots
[Y_{i_{k-1}}, Y_{i_k - 1}].
\end{equation*}
Let $Z$ be an optimal family of $d$-tuples of separators that stab
all $(k-1)$-chains of separators. We have $|Z| =
\z^{(d)}_{k-1}(\ell-1)$.

There must be a $d$-tuple $z = (Y_{a_1}, \ldots, Y_{a_d}) \in Z$
that stabs $\mathcal B'$. Therefore, there exist $d+1$ blocks $B'_0,
\ldots, B'_d \in \mathcal B$ such that the order between them and
the elements of $z$ is
\begin{equation*}
B'_0, Y_{a_1}, B'_1, Y_{a_2}, \ldots, Y_{a_d}, B'_d.
\end{equation*}
Let $q'_i$ be the point of $Q$ that lies in block $B'_i$, for $0\le
i\le d$.

Translate the $d$-tuple $z$ into a point $z' = (z'_1, \ldots, z'_d)
\in\R^d$ such that, for each $1\le i\le d$, the coordinate $z'_i$
lies between the $i$-th coordinates of the two points $y_{a_i}$,
$y'_{a_i}$ that constitute $Y_{a_i}$.

Then, since $z'$ is far from each of $q'_0, \ldots, q'_d$, it
follows from Lemmas~\ref{lemma_Carath_stair} and
\ref{lemma_conv_sconv} that $z'\in \conv \{q'_0, \ldots, q'_d\}
\subseteq \conv \diag'$.

Thus, the set $Z'\subset \R^d$ of all these points $z'$ for every
$z\in Z$ is a weak $1\over r$-net for $\diag$, and it has the
desired size.
\end{proof}

\begin{proof}[Proof of Theorem~\ref{t:alphas}(ii).]
Take $\ell = r\bigl(1 + P_d(\alpha(r)) \bigr)$, with $P_d$ as in
Theorem \ref{thm_z_j_upper}. Then
\begin{equation*}
\z^{(d)}_{\ell/r - 1}(\ell) = \z^{(d)}_{P_d(\alpha(r))}(\ell) \le
c''_d \ell \alpha_{\alpha(r)}(\ell),
\end{equation*}
which can be shown to be at most $4c''_d \ell$ by a simple argument.
\end{proof}

\section{Properties of stair-convexity and the transference lemma}
\label{sec_proofs}

In this section we prove Lemmas~\ref{lemma_Carath_stair} and
\ref{lemma_conv_sconv}, and then we use them to prove
Lemma~\ref{l:transfer} (the transference lemma). Along the way, we
establish other basic properties of stair-convexity.

Let us first introduce some notation. For a real number $y$ let
$h(y)$ denote the ``horizontal'' hyperplane $\{x\in\R^d: x_d=y\}$.
For a horizontal hyperplane $h=h(y)$ let $h^+:=\{x\in\R^d: x_d\ge
y\}$ be the upper closed half-space bounded by $h$, and let $h^-$ be
the lower closed half-space. For a set $S\subseteq\R^d$ let $S(y):=
S\cap h(y)$ be the horizontal slice of~$S$.

For a point $x=(x_1,\ldots,x_d)\in \R^d$ let $\overline x := (x_1,
\ldots, x_{d-1})$ be the projection of $x$ into $\R^{d-1}$, and
define $\overline S$ for $S \subset \R^d$ similarly. For a point
$x\in \R^{d-1}$ and a real number $x_d$, let $x \times x_d := (x_1,
\ldots, x_{d-1}, x_d)$, with a slight abuse of notation.

If $P$ and $Q$ are subsets of $\R^d$, we say that $P$ and $Q$
\emph{share the $i$-th coordinate} if $p_i = q_i$ for some $p\in P$,
$q\in Q$. Similarly, if $p\in \R^d$ and $Q\subset \R^d$, then we say
that $p$ and $Q$ share the $i$-th coordinate if $\{p\}$ and $Q$ do
so.

We begin with an equivalent, and perhaps somewhat more intuitive,
description of stair-convex sets.

\begin{lemma}\label{lemma_stconv}
A set $S \subseteq \R^d$ is stair-convex if and only if the
following two conditions hold:
\begin{enumerate}
\item[{\rm (SC1)}]
For every $y\in\R$, the set $\overline{S(y)}$ is a
$(d-1)$-dimensional stair-convex set.
\item[{\rm (SC2)}] {\rm (Slice-monotonicity)}
For every $y_1,y_2\in \R$ with $y_1\le y_2$ and
$S(y_2)\ne\emptyset$, we have $\overline{S(y_1)} \subseteq
\overline{S(y_2)}$.
\end{enumerate}
\end{lemma}

\begin{proof}
First let $S$ be stair-convex. Condition (SC1) is clear from the
definition of a stair-path. As for (SC2), we need to prove that for
every $a=(a_1,\ldots,a_d)\in S(y_1)$ the point
$a':=(a_1,\ldots,a_{d-1},y_2)$ directly above $a$ lies in $S(y_2)$.
But since $S(y_2)\ne\emptyset$, we can fix some $b\in S(y_2)$, and
then $a'$ lies on the stair-path $\sigma(a,b)$ and so $a'\in S(y_2)$
indeed.

Conversely, let $S\subseteq\R^d$ satisfy (SC1) and (SC2), and
let $a=(a_1,\ldots,a_d),b=(b_1,\ldots,b_d)\in S$ with $a_d\le
b_d$. Letting $a':=(a_1,\ldots,a_{d-1},b_d)$ be the point
directly above $a$ at the height of $b$ as in the definition of
the stair-path $\sigma(a,b)$, we have $\sigma(a',b)\subseteq S$
by the stair-convexity of $\overline{S(b_d)}$ and $aa'\subseteq
S$ by (SC2).
\end{proof}

\begin{lemma}\label{lemma_sconv}
The stair-convex hull of a set  $X\subseteq \R^d$ can be
(recursively) characterized as follows: For every horizontal
hyperplane $h=h(y)$ that does not lie entirely above $X$, let $X'$
stand for the vertical projection of $X\cap h^-$ into $h$. Then
$h\cap\sconv(X)=\sconv(X')$ (where $\sconv(X')$ is a stair-convex
hull in dimension $d-1$).
\end{lemma}

\begin{proof}
First we prove the inclusion $\sconv(X')\subseteq h\cap\sconv(X)$.
Let us fix a point $x_0\in X\cap h^+$ (i.e., above $h$ or on it),
and let $x$ be an arbitrary point of $X\cap h^-$. Then $x'$, the
vertical projection of $x$ into $h$, lies on the stair-path
$\sigma(x,x_0)$, and thus $X'\subseteq h\cap\sconv(X)$. Since
$h\cap\sconv(X)$ is stair-convex (by (SC1) in
Lemma~\ref{lemma_stconv}), we also have $\sconv(X')\subseteq
h\cap\sconv(X)$.

To establish the reverse inclusion, it suffices to show that for
every $(d-1)$-dimensional stair-convex $S'\subseteq h$ that contains
$X'$ there is a $d$-dimensional stair-convex set $S$ with $S\cap
h=S'$ that contains $X$. Such an $S$ can be defined as
$(\R^d\setminus h^-)\cup P^-(S')$, where
$P^-(S')=\{(x_1,\ldots,x_d)\in\R^d: (x_1,\ldots,x_{d-1},y)\in S',
x_d\le y\}$ is the semi-infinite vertical prism obtained by
extruding $S'$ downwards.
\end{proof}

Next, we prove Lemma~\ref{lemma_Carath_stair}, which asserts that a
point $x$ lies in the stair-convex hull of a set $X$ if and only if
$X$ contains a point of type $j$ with respect to~$x$ for every
$j=0,1,\ldots,d$.

\begin{proof}[Proof of Lemma~\ref{lemma_Carath_stair}. ]
Both directions follow by induction on $d$. The case $d=1$ is
trivial,  and so we assume $d\ge 2$.

Let $h$ be the horizontal hyperplane containing $x$. First we
suppose  $x \in \sconv(X)$. There exists a point $p_d\in X$
whose last coordinate is at least as large as that of $x$, and
this $p_d$ has type $d$ with respect to $x$.

Next, let $X'$ be the vertical projection of $X\cap h^-$ into
$h$ as in Lemma~\ref{lemma_sconv}. By that lemma we have $x\in
\sconv(X')$, and so, by induction, $X'$ contains points $p'_0,
\ldots, p'_{d-1}$ (not necessarily distinct) of types $0,
\ldots, d-1$, respectively, with respect to $x$. The
corresponding points $p_0, \ldots, p_{d-1} \in X$ also have
types $0, \ldots, d-1$ with respect to~$x$.

For the other direction, we suppose that there are points $p_0,
\ldots, p_d \in X$ of types $0, \ldots, d$ with respect to $x$.
Then the vertical projections of $p_0, \ldots, p_{d-1}$ into
$h$ also have types $0,\ldots, d-1$ w.r.t.\ $x$, and so by the
inductive hypothesis, their stair-convex hull contains $x$.
Since $p_d\in h^+$, it follows, again by
Lemma~\ref{lemma_sconv}, that $x\in\sconv(\{ p_0, \ldots,
p_d\})$.
\end{proof}

In order to prove Lemma~\ref{lemma_conv_sconv}, we first
establish some more properties of stair-convex hulls.

\begin{lemma}\label{lemma_share_coords}
Let $Q$ be a $k$-point set in $\R^d$ for some $k\le d+1$, and let
$p$ be a point in $\sconv(Q)$. Then $p$ shares at least $d-k+1$
coordinates with $Q$.
\end{lemma}

\begin{proof}
By induction on $d$. Let $q$ be the highest point of $Q$. First
suppose $p_d = q_d$. Then, $p$ shares the last coordinate with $Q$.
Further, by induction, $\overline p$ shares at least $d-k$
coordinates with $\overline Q$, so $p$ also shares at least $d-k$
out of the first $d-1$ coordinates with $Q$, and we are done.

Next suppose $p_d < q_d$. In this case, let $Q' = Q \setminus
\{q_d\}$. Then, by Lemma~\ref{lemma_sconv} and by induction,
$\overline p$ shares at least $d-(k-1)$ coordinates with
$\overline{Q'}$, so the same is true of $p$ and $Q'$.
\end{proof}

Kirchberger's theorem \cite{kirch} (see also \cite[p.~13]{matou_DG})
states that if $P$ and $Q$ are point sets in $\R^d$ such that
$\conv(P)$ and $\conv(Q)$ intersect, then there exist subsets
$P'\subseteq P$ and $Q'\subseteq Q$ of total size $|P| + |Q| \le
d+2$ such that $\conv(P')$ and $\conv(Q')$ intersect. The following
is an analogous result for stair-convex sets.

\def\toofew{(a)}
\def\justenough{(b)}
\def\toomany{(c)}

\begin{lemma}\label{lemma_kirch_stconv}
Let $P, Q\subset \R^d$ be two finite point sets that do not share
any coordinate, with $|P| = s$ and $|Q| = t$. Then:
\begin{enumerate}
\item[\rm\toofew]
If $s+t < d+2$, then $\sconv(P)$ and $\sconv(Q)$ do not intersect.

\item[\rm\justenough]
If $s+t = d+2$ and $\sconv(P)$, $\sconv(Q)$ intersect, then they do
so at a single point. Furthermore, the two highest points of $P \cup
Q$ belong one to $P$ and one to $Q$.

\item[\rm\toomany]
If $s+t \ge d+2$ and $\sconv(P)$, $\sconv(Q)$ intersect, then there
exist subsets $P' \subseteq P$, $Q' \subseteq Q$ of total size $|P'|
+ |Q'| = d+2$, such that $\sconv(P')$, $\sconv(Q')$ intersect.
\end{enumerate}
\end{lemma}

\begin{proof}
Let us first prove parts \toofew\ and \justenough. Suppose
$\sconv(P)$, $\sconv(Q)$ intersect, and let $x\in \R^d$ belong to
their intersection. If $s+t < d+2$, then, since $P$ and $Q$ do not
share any coordinate, Lemma~\ref{lemma_share_coords} implies that
$x$ shares a total of at least $d+1$ coordinates with $P \cup Q$,
which is impossible. Part \toofew\ follows.

Now suppose $s+t = d+2$; then, Lemma~\ref{lemma_share_coords}
implies that $x$ shares all $d$ coordinates with $P \cup Q$. If
there were another point $x'$ in $\sconv(P) \cap \sconv(Q)$, then
the same would be true of every point in $\sigma(x, x')$, which is
impossible. Let $a$, $a'$ be the two highest points of $P \cup Q$.
If they both belonged to $P$, say, then $\sconv( P \setminus \{a\}
)$ and $\sconv(Q)$ would still intersect, contradicting part
\toofew. Thus, part \justenough\ follows.

We now prove part \toomany\ by induction on $d$. The case $d=1$ is
clear, so let $d\ge 2$. Suppose $s+t\ge d+2$, and let $p_{\tp}$ and
$q_{\tp}$ be the highest points of $P$ and $Q$, respectively.

Let $x\in\sconv(P)\cap\sconv(Q)$, and let $h=h(x_d)$ be the
horizontal hyperplane containing $x$. Then, with $P^-:=P\cap h^-$
and $Q^-:=Q\cap h^-$, we have $\overline x\in \sconv(\overline
{P^-}) \cap\sconv(\overline{Q^-})$ by Lemma~\ref{lemma_sconv}.

By the inductive hypothesis there are subsets $P^-_0\subseteq P^-$
and $Q^-_0\subseteq Q^-$ with $|P^-_0|+|Q^-_0|=d+1$ and
$\sconv(\overline{P^-_0})\cap \sconv(\overline{Q^-_0})\ne\emptyset$.
Let $r\in\R^{d-1}$ be a point in this intersection (it need not be
identical to $\overline x$).

Let $a$ and $b$ be the highest points of $P^-_0$ and $Q^-_0$,
respectively. If $a$ lies below $b$, we set
$P_0:=P^-_0\cup\{p_{\tp}\}$ and $Q_0:=Q^-_0$; it is easily checked
that $r\times b_d\in \sconv(P_0)\cap \sconv(Q_0)$. Finally, if $a$
lies above $b$, we set $P_0:=P^-_0$ and $Q_0:=Q^-_0\cup\{q_{\tp}\}$;
then $r\times a_d\in \sconv(P_0)\cap \sconv(Q_0)$.
\end{proof}

We are now ready to prove Lemma~\ref{lemma_conv_sconv}, which states
that if $P, Q \subset \B$ are far apart, then $\sconv(P)$ and
$\sconv(Q)$ intersect if and only if $\conv(P)$ and $\conv(Q)$ do
so.

\begin{proof}[Proof of lemma~\ref{lemma_conv_sconv}]
Let $P,Q\subset\B$ be point sets that are far apart. Then, in
particular, $P$ and $Q$ do not share any coordinate (as is assumed
in Lemma~\ref{lemma_kirch_stconv}).

First we prove that $\sconv(P) \cap \sconv(Q) \neq \emptyset$
implies $\conv(P) \cap \conv(Q) \neq \emptyset$. The proof proceeds
by induction on $d$. In the inductive step we discard the point of
the largest height, and find an intersection of the convex hulls of
the remaining sets, which is a $(d-1)$-dimensional situation. Then
we would like to use the the discarded point for adjusting the last
coordinate of the intersection. In order to make this last step
work, instead of simply discarding the highest point, we use it to
perturb the other points.

Let us proceed in detail. By Lemma~\ref{lemma_kirch_stconv}\toomany,
we may assume that $P=\{p_1,\dotsc,p_s\}$ and
$Q=\{q_1,\dotsc,q_t\}$, with $s+t=d+2$. Let $y$ be a point in
$\sconv(P) \cap \sconv(Q)$. Let $p_s$, $q_t$ be the highest points
in $P$, $Q$, respectively, and let us assume $q_t$ lies above $p_s$.
By Lemma~\ref{lemma_kirch_stconv}\justenough, the set $Q^* := \{q_1,
\ldots, q_{t-1} \}$ lies below $p_s$.

We show by induction on $d$ that the following system of equations
and inequalities with unknowns $a_1,\dotsc,a_s$, $b_1,\dotsc,b_t$
has a solution:
\begin{subequations}
\begin{align}
a_1+\dotsb+a_s&=b_1+\dotsb+b_t, \label{eq_a_eq_b}=1\\
a_1 p_1+\dotsb+a_s p_s&=b_1 q_1+\dotsb+b_t q_t, \label{eq_cc_p_cc_q}\\
a_1,\dotsc,a_s,b_1,\dotsc,b_t&\geq 1/x_{dm}\label{conv_sconv_ineq},
\end{align}
\end{subequations}
where, as  we recall, $x_{dm}$ is the maximum height of a point in
$\B$. Equations (\ref{eq_a_eq_b}) and (\ref{eq_cc_p_cc_q}) assert
that $\conv(P)\cap\conv(Q)\neq \emptyset$, and the inequalities
(\ref{conv_sconv_ineq}) are crucial in the induction argument.

The case $d=1$ is an easy computation, and so we assume $d\geq 2$.

Let $\alpha>0$ be a parameter, and for each $q_i\in Q^*$ define the
``perturbed" point $q_i'=(1-\alpha)q_i+\alpha q_t$, which lies on
the segment $q_i q_t$. Let $Q'=\{q_1',\dotsc,q_{t-1}'\}$. Since
$q_t$ is very high above $Q^*$, the segments $q_i q_t$ are ``almost"
vertical. As we will see, it is possible to choose the parameter
$\alpha$ large enough so that $Q'$ lies above $p_s$, and yet small
enough so that $\overline{Q'}$ is not ``too far" from
$\overline{Q^*}$.

Specifically, we will choose $\alpha \in [1/x_{dm},1/2 x_{(d-1)m}]$.
We claim that for any such choice of $\alpha$, the set
$\overline{Q'}$ is far apart from $\overline{P}$ and
$\sconv(\overline{Q'})$ intersects $\sconv(\overline{P})$ iff
$\sconv(\overline{Q^*})$ does.

To see this, recall that $A\myll i B$ means $K_i A \le B$ with
$K_i=2^d x_{(i-1)m}$. To indicate the dependence on $d$, we
temporarily adopt the more verbose notation $A \ll_{i,d} B$ in place
of $A \myll i B$. Since $\alpha \le 1/2 x_{(d-1)m}$, for every
$i=1,2,\ldots,t$ and every $k=1,2,\ldots,d-1$ we have $\alpha
q_{ik}\leq \tfrac{1}{2}$, and so
\begin{equation}\label{eq_ordpreserv}
\begin{aligned}
q_{ik} \ll_{k,d} p_{jk} &\implies q_{ik}' = (1-\alpha)q_{ik}+\alpha
q_{tk}\leq q_{ik}+\tfrac{1}{2}\leq
 2 q_{ik} \ll_{k,d-1} p_{jk},\\
\text{and}\qquad q_{ik} \gg_{k,d} p_{jk} &\implies
q_{ik}'=(1-\alpha)q_{ik}+\alpha q_{tk}\geq q_{ik}-\tfrac{1}{2}\geq
\tfrac{1}{2}q_{ik} \gg_{k,d-1} p_{jk}.
\end{aligned}
\end{equation}
So indeed, there is no combinatorial change between $\overline{
Q^*}$ and $\overline{Q'}$ as far as intersection of stair-convex
hulls with $\sconv(\overline{P})$ is concerned.

Further, since $q_t$ is the highest point in $P \cup Q$, we have
$\overline y \in\sconv(\overline P) \cap \sconv(\overline{ Q^*})$,
and therefore, $\sconv(\overline P) \cap \sconv(\overline{ Q'})
\neq\emptyset$. Hence, by the induction hypothesis there exists a
point $r\in \conv( \overline P) \cap \conv(\overline{Q'})$ and
further, there exist real numbers
$a_1,\dotsc,a_s,b_1,\dotsc,b_{t-1}$ satisfying
\begin{align}
a_1+\dotsb+a_s&=b_1+\dotsb+b_{t-1}=1, \nonumber\\
a_1 \overline{p}_1+\dotsb+a_s \overline{p}_s &= b_1
\overline{q}'_1+\dotsb+b_{t-1} \overline{q}'_{t-1}=r,
 \label{eq_cc_r}\\
a_1,\dotsc,a_s,b_1,\dotsc,b_{t-1}&\geq 1/x_{(d-1)m}. \nonumber
\end{align}

Thus, there are real numbers $h_P$ and $h_Q$ for which $r\times
h_P\in \conv(P)$ and $r\times h_Q\in\conv(Q)$. Now we exploit the
freedom in choosing $\alpha$. First let $\alpha:=1/2x_{(d-1)m}$;
then we have $q'_{id} > \alpha q_{td}=q_{td}/2x_{(d-1)m} \ge p_{sd}$
(since $q_{td} \mygg{d} p_{sd}$). Thus, $Q'$ lies entirely above
$p_s$, implying that $h_Q>h_P$ in this case.

Next, let $\alpha:=1/x_{dm}$. Since $p_{sd} \mygg{d} q_{id}$ for
every $i=1,2,\ldots,t-1$, we have $a_s p_{sd}\geq
p_{sd}/x_{(d-1)m}>q_{id} + 1 \ge q_{id} + \alpha q_{td} > q'_{id}$.
Hence
\begin{equation*}
a_1 p_{1d} + \dotsb + a_s p_{sd} > a_s p_{sd} > b_1 q'_{1d} + \dotsb
+ b_{t-1} q'_{(t-1)d},
\end{equation*}
implying that $h_P>h_Q$ in this case.

Since solutions of linear equations depend continuously on the
coefficients, the intermediate value theorem implies that there is
an $\alpha$ in the interval $[1/x_{dm},1/2x_{(d-1)m}]$ for which
$h_Q=h_P$. Fix this $\alpha$. Then the point $r\times h_P=r\times
h_Q$ lies in both $\conv(P)$ and $\conv(Q)$, as desired. It remains
to verify the inequalities~\eqref{conv_sconv_ineq}. We have
\begin{align*}
r\times h_P&=a_1 p_1+\dotsb+a_s p_s,\\
r\times h_Q&=(1-\alpha)b_1 q_1+\dotsb+(1-\alpha)b_{t-1}
q_{t-1}+\alpha q_t.
\end{align*}
Then $a_i\geq 1/x_{dm}$ follows from $a_i\geq 1/x_{(d-1)m}$; the
inequality $(1-\alpha)b_i\geq 1/x_{dm}$ follows since $1-\alpha\geq
1/2$ and by the definition of $x_{dm}$; and $\alpha\geq 1/x_{dm}$
holds by our very choice of $\alpha$. The first implication in
Lemma~\ref{lemma_conv_sconv} is proved.

\medskip

We now tackle the reverse implication, again proceeding by induction
on $d$. Let us suppose that $\conv(P) \cap \conv(Q)\neq\emptyset$.
By Kirchberger's theorem, we can assume that $\abs{P}+\abs{Q}\le
d+2$. As above let $P=\{p_1,\dotsc,p_s\}$, $Q=\{q_1,\dotsc,q_t\}$,
with points $p_s$ and $q_t$ highest in their respective sets, and
assume $q_t$ lies above $p_s$.

Let $r\in \conv(P)\cap \conv(Q)$. Then there exist nonnegative
coefficients $a_1, \dotsc a_s$, $b_1, \dotsc, b_{t-1}$, and $\alpha$
that satisfy
\begin{align*}
r = a_1 p_1+\dotsb+a_s p_s&=b_1 q_1+\dotsb+b_{t-1} q_{t-1} + \alpha q_t,\\
a_1+\dotsb+a_s&=b_1+\dotsb+b_{t-1} + \alpha=1.
\end{align*}
Since $\sum a_i p_{id}\leq p_{sd}$, it follows that $\alpha\leq
p_{sd}/q_{td}\leq 1/2x_{(d-1)m}$.

As in the proof of the first implication,  let
$Q^*:=\{q_1,\ldots,q_{t-1}\}$, let $q'_i = (1-\alpha) q_i + \alpha
q_t$ (with the $\alpha$ just introduced), and let $Q' = \{q'_1,
\dotsc, q'_{t-1} \}$. Then $r$ is a convex combination of the points
in $Q'$, so $r\in\conv(Q')$.

Therefore, $\conv(\overline{P}) \cap \conv(\overline{Q'}) \neq
\emptyset$, and so by the induction hypothesis
$\sconv(\overline{P})\cap\sconv(\overline{Q'})\neq\emptyset$. But
arguing again as in \eqref{eq_ordpreserv}, the order of points of
$\overline{Q'}$ with respect to $\overline{P}$ is same as that of
$\overline{Q^*}$ with respect to $\overline{P}$ in each coordinate;
therefore $\sconv(\overline{P})$ and $\sconv(\overline{Q^*})$ must
also intersect; let $y\in\R^{d-1}$ belong to their intersection. We
claim that $p_s$ lies above $Q^*$; this is enough, since it implies
that $y\times p_{sd} \in \sconv(P) \cap \sconv(Q)$.

Suppose it is not the case, and that $q_{t-1}$, say, lies above
$p_s$. Since $r\in\conv(Q)$, there exists a point $q^\circ$ in the
segment $q_{t-1}q_t$ such that $r \in \conv(Q^\circ)$, where
$Q^\circ =\{q_1, \dotsc, q_{t-2}, q^\circ\}$.

Thus, $\conv(P)\cap \conv(Q^\circ)\ne\emptyset$. Since $q^\circ$
lies above $p_s$, we can apply the preceding argument with $Q^\circ$
in place of $Q$, and we infer that $\sconv(\overline{P}) \cap
\sconv( \overline{\{q_1, \dotsc, q_{t-2} \}})\ne\emptyset$. However,
these sets have a total of only $d$ points, contradicting
Lemma~\ref{lemma_kirch_stconv}\toofew.
\end{proof}

Next, we derive auxiliary results needed for the proof of
transference lemma (Lemma~\ref{l:transfer}). Given sets $P, Q
\subseteq \R^d$, we define the operation
\begin{equation*}
P \ominus Q := \{ p\in P: p+Q \subseteq P\},
\end{equation*}
where $p+Q=\{p+q:q\in Q\}$.

\begin{lemma}\label{l:addcube}
Let $S\subseteq \I^d$ be a stair-convex set, and let $\Gu = \Gu(m)$
be the uniform grid of side $m$. Then, for every $\delta>0$, the set
$S^{\delta-}:= S\ominus [0,\delta]^d$ is stair-convex,
\begin{equation*}
\vol(S^{\delta-})\ge \vol(S)-d\delta, \qquad \text{and}\qquad
|S^{\delta-}\cap\Gu|\ge |S\cap\Gu|- d\lceil(m-1)\delta\rceil
m^{d-1}.
\end{equation*}
\end{lemma}

\begin{proof}
For an index $i\in\{1,2,\ldots,d\}$ and $\delta>0$ let $s_i(\delta)$
be the initial closed segment of the positive $x_i$-axis (starting
at the origin) of length $\delta$.

We prove that for every stair-convex $S\subseteq \I^d$ and every
$\delta>0$ the set $S':=S\ominus s_i(\delta)$ is stair-convex, has
volume at least $\vol(S)-\delta$, and contains at least $|S\cap\Gu|-
\lceil(m-1)\delta\rceil m^{d-1}$ points of $\Gu$. The assertion of
the lemma then follows by $d$-fold application of this statement and
by noticing that $S\ominus [0,\delta]^d=S\ominus
s_1(\delta)\ominus\cdots \ominus s_d(\delta)$.

As for the stair-convexity of $S'$, the following actually holds: If
$S$ is stair-convex and $D$ is arbitrary, then $S\ominus D$ is
stair-convex too. This follows from the translation invariance of
stair-paths. Namely, $\sigma(a+x,b+x)=x+\sigma(a,b)$, and thus for
$a,b\in S\ominus D$ we have $a+x$ and $b+x$ in $S$ for all $x\in D$,
so $x+\sigma(a,b)=\sigma(a+x,b+x)\subseteq S$, and thus
$\sigma(a,b)\subseteq S\ominus D$.

The claim about $\vol(S')$ follows by Fubini's theorem, since
$S\setminus S'$ intersects every line parallel to the $x_i$-axis in
a single segment of length at most~$\delta$. The claim about the
number of grid points follows similarly, by noticing that the grid
$\Gu(m)$ has step $\frac1{m-1}$ and thus $S\setminus S'$ contains at
most $\lceil \delta(m-1) \rceil$ grid points on each line parallel
to the $x_i$-axis.
\end{proof}

\begin{corollary}[Grid approximation]\label{corol_sconv_contains_grid_pts}
Let $S\subseteq \I^d$ be a stair-convex set, and let $g_S = \abs{ S
\cap \Gu(m)}$ be the number of points of the uniform grid contained
in $S$. Then,
\begin{equation*}
\bigl| g_S -  (m-1)^d\vol(S)\bigr| \le dm^{d-1}.
\end{equation*}
\end{corollary}

\begin{proof}
Let $\delta:=\frac1{m-1}$ be the step of the grid $\Gu(m)$. For
every grid point $p\in \Gu \cap S^{\delta-}$, the cube
$p+[0,\delta]^d$ is contained in $S$, and since such cubes have
disjoint interiors, we have $\vol(S)\ge \delta^d
|S^{\delta-}\cap\Gu|\ge \delta^d g_S-\delta^d dm^{d-1}$ by the
second inequality in Lemma~\ref{l:addcube}. Multiplying by
$\delta^{-d}$ we get $g_S \le (m-1)^d\vol(S)+ dm^{d-1}$, one of the
inequalities in the corollary.

For the other inequality, we observe that if $p\in\Gu(m)$ is a grid
point such that the cube $p+[-\delta,0]^d$ intersects $S^{\delta-}$,
then $p\in S$. So using the first inequality of
Lemma~\ref{l:addcube} gives $\vol(S)\le \vol(S^{\delta-})+d\delta\le
\delta^d g_S + d\delta$, and we are done.
\end{proof}

\begin{proof}[Proof of Lemma~\ref{l:transfer}. ]
Let us prove part (i). So let $N$ a weak $\eps$-net for
$\Gs=\Gs(m)$, and let $s = |N|$.

Let us call a point $p\in \Gs$ \emph{good} if it is far apart from
every point of $N$; otherwise, $p$ is \emph{bad}. There are at most
$2ds m^{d-1}$ bad points in $\Gs$.

Let $\eps' := \eps + 2d(s+1)/m$, and let us consider a stair-convex
set $S'\subseteq \I^d$ of volume $\eps'$. By
Corollary~\ref{corol_sconv_contains_grid_pts}, $S'$ contains a set
$P'\subseteq \Gu$ of at least $\eps'(m-1)^d - dm^{d-1}$ grid points.

Let $P = \pi^{-1} (P')$ be the corresponding subset of $\Gs$. By
removing all bad points from $P$ we obtain a set $P^*$ of at least
$\eps' (m-1)^d - d(2s+1) m^{d-1} \ge \eps m^d$ good points. Since
$N$ is a weak $\eps$-net, there exists a point $x\in N\cap \conv
(P^*)$.

Since all points of $P^*$ are far apart from $x$, it follows by
Lemma~\ref{lemma_conv_sconv} that $x\in\sconv(P^*)$. Further, $\pi$
preserves order in each coordinate, so
\begin{equation*}
x':=\pi(x) \in \sconv\bigl(\pi(P^*)\bigr)\subseteq S'.
\end{equation*}
Since $x'\in \pi(N)$, this proves that $\pi(N)$ intersects every
stair-convex set of volume $\eps'$ in $\I^d$. This finishes the
proof of part (i) of the transference lemma.

Part (ii) is proved similarly, only with the roles of convexity and
stair-convexity interchanged.
\end{proof}

\section{Conclusion}\label{sec_conclusions}

In this paper we provide superlinear lower bounds for weak
$\frac1r$-nets, but the gaps between the known lower and upper
bounds for weak $\frac1r$-nets are still huge. The most significant
gaps are: between $\Omega(r\log r)$ and $O(r^2)$ for the general
planar case; between $\Omega{\bigl( r\log^{d-1} r\bigr)}$ and $O(r^d
\mathrm{polylog}\, r)$ for the general case in $\R^d$; and between
$\Omega(r)$ and $O(r\alpha(r))$ for planar point sets in convex
position.

The point set that allowed us to obtain the superlinear lower
bounds, the stretched grid, might be useful for further problems
too, especially since problems about convexity in the stretched grid
can be recast in purely combinatorial terms. One might ask, to what
extent the stretched grid is ``special'' as far as weak $\eps$-nets
are concerned.

On the one hand, it does provide stronger lower bounds than some
other sets. Namely, it is easy to show that there exist weak $\frac1r$-nets
of size $O(r\log r)$ for the uniform distribution in the $d$-dimensional
unit cube $\I^d$ (or for any sufficiently large finite uniformly
distributed set).\footnote{This is because there exist $\frac1r$-nets
of size $O(r\log r)$ with respect to ellipsoids, say, and
every convex set of volume $\eps$ contains an ellipsoid
of volume $\Omega(\eps)$ by the L\"owner--John theorem;
see, e.g., \cite{matou_DG} for background.}
Thus for $d\ge 3$ the stretched grids need strictly larger
weak $\frac1r$-nets than uniformly distributed sets.

On the other hand, we tend to believe that stretched grids are not
special in providing superlinear lower bounds: We conjecture that
\emph{no} sets in general position in $\R^d$, $d\ge 3$, admit
linear-size weak $\frac 1r$-nets. (More precisely: For every $C$
there exist $r$ and $n_0$, also possibly depending on $d$, such that
$f(X,r)\ge Cr$ for every $X\subset\R^d$ in general position and with
at least $n_0$ points.) This conjecture may be very hard to prove,
though, since it would also imply a superlinear lower bound for
$\frac 1r$-nets for geometrically defined set systems of bounded VC
dimension, which has been an outstanding problem in discrete
geometry for several decades.

\bibliographystyle{plain}
\bibliography{stair_convex_IJM}

\end{document}